\documentclass[twoside]{aiml20}

\usepackage{aiml20macro}

\usepackage{graphicx}
\usepackage{amsmath}
\usepackage{amssymb}
\usepackage{bussproofs}
\usepackage{dsfont}
\usepackage{soul}

\def\lastname{Rogozin}

\begin{document}

\begin{frontmatter}
  \title{The Distributive Full Lambek Calculus with Modal Operators}
  \author{Daniel Rogozin} \footnote{The research is supported by the Presidential Council, research grant MK-430.2019.1.}
  \address{Lomonosov Moscow State University \\ Moscow, Russia \\ Serokell O\"{U} \\ Tallinn, Estonia}

  \begin{abstract}
In this paper, we study logics of bounded distributive residuated lattices with modal operators considering $\Box$ and $\Diamond$ in a noncommutative setting.  We introduce relational semantics for such substructural modal logics. We prove that any canonical logic is Kripke complete via discrete duality and canonical extensions. That is, we show that a modal extension of the distributive full Lambek calculus is the logic of its frames if its variety is closed under canonical extensions. After that, we establish a Priestley-style duality between residuated distributive modal algebras and topological Kripke structures based on Priestley spaces.
  \end{abstract}

 \end{frontmatter}

\section{Introduction}

Substructural logic is logic lacking some of the well-known structural rules such as contraction, weakening, or exchange. Algebraically, substructural logics represent ordered residuated algebras \cite{galatos2007residuated}. One may examine modalities in substructural logics at least in two perspectives. Modalities in those logics are pretty helpful is such philosophical issues such as a relevant necessity. The second perspective is more applied and related to such issues as resource management. Here, linear logic proved its efficiency in computer science and linguistics.

The computer science applications are more related to resource-sensitive computation and
related applications in type theory \cite{abramsky1993computational} \cite{girard1987linear}.

The noncommutative version of linear logic, the Lambek calculus \cite{lambek1958mathematics}, has plenty of applications in linguistics, for instances, proof-theoretical characterisation of inference in Lambek grammars, the equivalent version of context-free grammars \cite{pentus1993lambek}. The Lambek calculus characterises inference in categorial grammars. There are several approaches to consider categorial grammars from a broader modal point of view as well, see the paper by van Benthem \cite{vanbenthem2003categorial}, for example.
Modalities in such logics (the $!$-modality, to be more precise) introduce lacking structural rules in a restricted way:

\begin{prooftree}
  \AxiomC{$\Gamma, \Delta \Rightarrow B$}
  \UnaryInfC{$\Gamma, ! A, \Delta \Rightarrow B$}
\end{prooftree}
\begin{minipage}{0.5\textwidth}
  \begin{flushleft}
    \begin{prooftree}
      \AxiomC{$\Gamma, ! A, B, \Delta \Rightarrow C$}
      \UnaryInfC{$\Gamma, B, ! A, \Delta \Rightarrow C$}
    \end{prooftree}
  \begin{prooftree}
  \AxiomC{$\Gamma, ! A, B, ! A, \Delta \Rightarrow C$}
  \UnaryInfC{$\Gamma, ! A, B, \Delta \Rightarrow C$}
  \end{prooftree}
  \end{flushleft}
\end{minipage}\hfill
\begin{minipage}{0.5\textwidth}
  \begin{flushright}
    \begin{prooftree}
      \AxiomC{$\Gamma, ! A, ! A, \Delta \Rightarrow B$}
      \UnaryInfC{$\Gamma, ! A, \Delta \Rightarrow B$}
    \end{prooftree}
  \begin{prooftree}
  \AxiomC{$\Gamma, ! A, B, ! A, \Delta \Rightarrow C$}
  \UnaryInfC{$\Gamma, B, ! A, \Delta \Rightarrow C$}
  \end{prooftree}
  \end{flushright}
\end{minipage}

  \vspace{\baselineskip}

Modal extensions of linear logic have an interpretation within the context of resource management based on the phase semantics proposed by Girard \cite{girard1995linear}. Algebraically, exponential modalities were studied by Ono as additional exponential operators on FL algebras \cite{ono1993semantics}. The Lambek calculus and its modal extensions also have the cover semantics proposed by Goldblatt \cite{goldblatt2011grishin} \cite{goldblatt2016cover}.
The abstract polymodal case of such an extension of the full Lambek calculus was recently studied by Kanovich, Kuznetsov, Scedrov, and Nigam \cite{kanovich2019subexponentials}. In this paper, (sub)exponential modalities are considered from proof-theoretical and complexity perspectives.

The basic Lambek calculus is complete with respect to so-called language models, residuated semigroups on subsets of free semigroup \cite{PENTUS1995179}. The Lambek calculus is complete with respect to subsets of a transitive relation as well, see \cite{Andreka1994}. Dunn, Gehrke, Palmigiano, and other authors considered relational semantics for the Lambek calculus using canonical extensions \cite{chernilovskaya2012generalized} \cite{dunn2005}. Alternatively, one may consider bi-approximation semantics for substructural logic studied by Suzuki \cite{Suzuki2010-SUZBSF-3}.

In this paper, we study the distributive version of the full Lambek calculus extended with normal modal operators $\Box$ and $\Diamond$ to consider a broader class of noncommutative modalities as an abstraction of storage operators in noncommutative linear logic in a distributive setting. We introduce noncommutative Kripke frames, relational structures for the distributive Lambek calculus extended with binary modal relations. We establish a discrete duality between such Kripke frames and perfect distributive residuated modal algebras developing an approach proposed in \cite{Distr}. After that, we overview canonical extensions of related modal algebras applying techniques provided in \cite{dunn2005} \cite{gehrketopological} \cite{Gehrke_Jonsson_2004} to show that any canonical residuated distributive modal logic is Kripke complete. We also prove that the subexponential modality axioms are canonical ones. Thus, we show that the corresponding logics enriched with subexponentials are Kripke complete. Finally, we extend the obtained duality to topological duality between residuated distributive modal algebras and special topological Kripke spaces based on bDRL-spaces, Priestley spaces with a ternary relation dual algebras of which are bounded distributive residuated lattices \cite{galatos2003varieties}. We also use some ideas from positive modal and intuitionistic modal logics \cite{celani1999priestley} \cite{palmigiano2004dualities}.

The text contains a short appendix with the brief survey on canonical extensions and duality for bounded distributive lattices to keep the paper self-contained.
\section{The distributive Lambek calculus with modal operators}
In this section, we formulate the distributive full Lambek calculus enriched with modal operators. We represent such logics with pairs of formulas that have the form $\varphi \vdash \psi$. One needs to read that as ``$\psi$ is derivable from $\varphi$''.

The language extends the language of the full Lambek calculus with modal operators $\Box$ and $\Diamond$ as follows:

$\varphi, \psi ::= \bot \:
   | \: \top \:
   | \: p \:
   | \: (\varphi \bullet \psi) \:
   | \: (\varphi \setminus \psi) \:
   | \: (\varphi / \psi) \:
   | \: (\varphi \lor \psi) \:
   | \: (\varphi \land \psi) \:
   | \: \Box \varphi \:
   | \: \Diamond \varphi$.

By a \emph{substructural distributive normal modal logic}, we mean some set of pairs $\vartheta \vdash \psi$, where $\vartheta$, $\psi$ are formulas, according to the following definition:
\begin{definition} A substructural normal distributive modal logic is a set of sequents $\Lambda$ that contains the axioms below and closed under the following inference rules:
\begin{center}
  $p \vdash p$
\end{center}
\begin{minipage}{0.5\textwidth}
  \begin{flushleft}
\begin{itemize}
  \item $\bot \vdash p$
  \item $p \vdash \top$
  \item $p_i \vdash p_1 \lor p_2, i = 1, 2$
  \item $p_1 \land p_2 \vdash p_i, i = 1, 2$
  \item $p \land (q \lor r) \vdash (p \land q) \lor (p \land r)$
  \item $(p \bullet q) \bullet r \dashv \vdash p \bullet (q \bullet r)$
\end{itemize}
\end{flushleft}
\end{minipage}\hfill
\begin{minipage}{0.5\textwidth}
  \begin{flushright}
\begin{itemize}
  \item $p \bullet {\bf 1} \dashv \vdash  {\bf 1} \bullet p \dashv \vdash p$
  \item $\Diamond (p \lor q) \vdash \Diamond p \lor \Diamond q$
  \item $\Diamond \bot \vdash \bot$
  \item $\Box p \land \Box q \vdash \Box (p \land q)$
  \item $\top \vdash \Box \top$
  \item $\Box p \bullet \Box q \vdash \Box (p \bullet q)$
\end{itemize}
\end{flushright}
\end{minipage}

\begin{minipage}{0.5\textwidth}
  \begin{flushleft}
    \begin{prooftree}
    \AxiomC{$\varphi \vdash \psi$}
    \AxiomC{$\psi \vdash \theta$}
    \BinaryInfC{$\varphi \vdash \theta$}
  \end{prooftree}

  \begin{prooftree}
    \AxiomC{$\varphi \vdash \psi$}
    \AxiomC{$\theta \vdash \psi$}
    \BinaryInfC{$\varphi \lor \theta \vdash \psi$}
  \end{prooftree}

  \begin{prooftree}
    \AxiomC{$\varphi \bullet \theta \vdash \psi$}
    \doubleLine
    \UnaryInfC{$\theta \vdash \varphi \setminus \psi$}
  \end{prooftree}

  \begin{prooftree}
    \AxiomC{$\varphi \vdash \psi$}
    \UnaryInfC{$\Box \varphi \vdash \Box \psi$}
  \end{prooftree}
  \end{flushleft}
\end{minipage}\hfill
\begin{minipage}{0.5\textwidth}
  \begin{flushright}
    \begin{prooftree}
    \AxiomC{$\varphi(p) \vdash \psi(p)$}
    \UnaryInfC{$\varphi [p := \gamma] \vdash \psi [p := \gamma]$}
  \end{prooftree}

  \begin{prooftree}
    \AxiomC{$\varphi \vdash \psi$}
    \AxiomC{$\varphi \vdash \theta$}
    \BinaryInfC{$\varphi \vdash \psi \land \theta$}
  \end{prooftree}

  \begin{prooftree}
    \AxiomC{$\theta \vdash \varphi \setminus \psi$}
    \doubleLine
    \UnaryInfC{$\varphi \bullet \theta \vdash \psi$}
  \end{prooftree}

  \begin{prooftree}
    \AxiomC{$\varphi \vdash \psi$}
    \UnaryInfC{$\Diamond \varphi \vdash \Diamond \psi$}
  \end{prooftree}
  \end{flushright}
\end{minipage}
\end{definition}

Substructural normal distributive modal logic extends distributive normal modal logic with residuals, product, and the axiom connecting $\Box$ and $\bullet$. See this paper \cite{Distr} to examine in-depth logics of bounded distributive lattices with modal operators that we extend.

We introduce ternary Kripke frames with the additional binary modal relations. Such a ternary frame might be considered as a noncommutative generalisation of a relevant Kripke, see, e.g., here \cite{Seki2003}. As it is usual in the relational semantics of substructural logic, product and residuals have the ternary semantics as in, e.g., \cite{Andreka1994} \cite{doi:10.1002/malq.19920380113} \cite{routley1973semantics}.

\begin{definition}
  A Kripke frame is a structure $\mathcal{F} = \langle W, \leq, R, R_{\Box}, R_{\Diamond}, \mathcal{O} \rangle$, where $\langle W, \leq, \rangle$ is a partial order, $R$ is a ternary relation on $W$, $R_{\Box}, R_{\Diamond}$ are binary relations on $W$, and $\mathcal{O} \subseteq W$ such that for all $\forall u, v, w, u', v', w' \in W$
  \begin{enumerate}
    \item $R u v w \: \& \: w R_{\Box} w' \Rightarrow \exists x, y \in W \:\:
    R x y w' \: \& \: u R_{\Box} x \: \& \: v R_{\Box} y$
    \item $\exists x \in W (R u w x \: \& \: R x u' v') \Leftrightarrow \exists y \in W (R w u' y \: \& \: R u y v')$
    \item $R u v w \: \& \: u' \leq u \Rightarrow R u' v w$, $R u v w \: \& \: v' \leq v \Rightarrow R u v' w$, $R u v w \: \& \: w \leq w' \Rightarrow R u v w'$
    \item $\forall o \in \mathcal{O} \:\: R v o w \Leftrightarrow R o v w $
    \item $v \leq w \Leftrightarrow \exists o \in \mathcal{O} R v o w$
    \item $\mathcal{O}$ is upwardly closed
    \item $u \leq v \: \& \: v R_{\Box} w \Rightarrow u R_{\Box} w$ and $u \leq v \: \& \: u R_{\Diamond} w \Rightarrow v R_{\Diamond} w$
  \end{enumerate}
\end{definition}

A Kripke model is a Kripke frame equipped with a valuation function that maps each propositional variable to
$\leq$-upwardly closed subset of worlds.

\begin{definition} Let $\mathcal{F} = \langle W, \leq, R, R_{\Box}, R_{\Diamond}, \mathcal{O} \rangle$ be a Kripke frame, a Kripke model is a pair $\mathcal{M} = \langle \mathcal{F}, \vartheta \rangle$, where $\vartheta : \operatorname{PV} \to \operatorname{Up}(W, \leq)$. Here, $\operatorname{Up}(W, \leq)$ is the collection of all upwardly closed sets. The connectives have the following semantics:

  \begin{enumerate}
    \item $\mathcal{M}, w \models p \Leftrightarrow w \in \vartheta(p)$
    \item $\mathcal{M}, w \models \top$, $\mathcal{M}, w \not\models \bot$, $\mathcal{M}, w \models {\bf 1} \Leftrightarrow w \in \mathcal{O}$
    \item $\mathcal{M}, w \models \varphi \bullet \psi \Leftrightarrow \exists u,v \in W \: R u v w \: \& \: \mathcal{M}, u \models \varphi \: \& \: \mathcal{M}, v \models \psi$
    \item $\mathcal{M}, w \models \varphi \setminus \psi \Leftrightarrow \forall u,v \in W \: R u w v \: \& \: \mathcal{M}, u \models \varphi \text{ implies } \mathcal{M}, v \models \psi$
    \item $\mathcal{M}, w \models \psi / \varphi \Leftrightarrow \forall u,v \in W \: R u w v \: \& \: \mathcal{M}, u \models \varphi \text{ implies } \mathcal{M}, v \models \psi$
    \item $\mathcal{M}, w \models \varphi \land \psi \Leftrightarrow \mathcal{M}, w \models \varphi \: \& \: \mathcal{M}, w \models \psi$
    \item $\mathcal{M}, w \models \varphi \lor \psi \Leftrightarrow \mathcal{M}, w \models \varphi \: \text{or} \: \mathcal{M}, w \models \psi$
    \item $\mathcal{M}, w \models \Box \varphi \Leftrightarrow \forall v \in R_{\Box}(w) \:\: \mathcal{M}, v \models \varphi$
    \item $\mathcal{M}, w \models \Diamond \varphi \Leftrightarrow \exists v \in R_{\Diamond}(w)
    \:\: \mathcal{M}, v \models \varphi$
    \item $\mathcal{M}, w \models \varphi \vdash \psi \Leftrightarrow \mathcal{M}, w \models \varphi \Rightarrow \mathcal{M}, w \models \psi$
    \item $\mathcal{M} \models \varphi \vdash \psi \Leftrightarrow \forall w \in W \: \mathcal{M}, w \models \varphi \vdash \psi$
  \end{enumerate}
\end{definition}

The following definitions are also standard ones.

\begin{definition}\label{frame}
  Let $\mathcal{F}$ be a Kripke frame
  \begin{enumerate}
    \item Let $\varphi \vdash \psi$ be a sequent, then $\mathcal{F} \models \varphi \vdash \psi$ iff for all valuations $\vartheta$ $\langle \mathcal{F}, \vartheta \rangle \models \varphi \vdash \psi$.
    \item $\operatorname{Log}(\mathcal{F}) = \{ \varphi \vdash \psi \: | \: \mathcal{F} \models \varphi \vdash \psi \}$.
    \item Let $\mathbb{F}$ be a class of Kripke frames, then $\operatorname{Log}(\mathbb{F}) = \bigcap \limits_{\mathcal{F} \in \mathbb{F}} \operatorname{Log}(\mathcal{F})$.
    \item Let $\Lambda$ be a substructural normal modal logic, then $\operatorname{Frames}(\mathcal{L}) = \{ \mathcal{F} \: | \: \mathcal{F} \models \mathcal{L} \}$ and $\mathcal{L}$ is \emph{complete} iff $\mathcal{L} = \operatorname{Log}(\operatorname{Frames}(\mathcal{L}))$.
  \end{enumerate}
\end{definition}
By ${\bf L}_{\bf K}$, we mean the minimal substructural distributive normal modal logic, the minimal set of sequents containing the axioms above and closed under the required inference rules.

The soundness theorem is the standard one.

\begin{theorem}
  Let $\mathbb{F}$ be a class of Kripke frames, then $\operatorname{Log}(\mathbb{F})$ is a substructural distributive normal modal logic.
\end{theorem}

\begin{proof} Let us show that $\mathcal{F} \models \Box p \bullet \Box q \Rightarrow \Box (p \bullet q)$, where $\mathcal{F} = \langle W, R, R_{\Box}, R_{\Diamond}, \mathcal{O}\rangle$ is an arbitrary frame. Let $\vartheta$ be a valuation and $\mathcal{M} = \langle \mathcal{F}, \vartheta \rangle$ a model such that $\mathcal{M}, w \models \Box p \bullet \Box q$, and $w \in W$.
Then there exists $u, v$ such that $R u v w$, $\mathcal{M}, u \models \Box p$, $\mathcal{M}, v \models \Box q$.
On the other hand, let $w' \in R_{\Box}(w)$. $R u v w \: \& \: w R_{\Box} w'$ implies $\exists x, y \in W \:\: R x y w' \: \& \: u R_{\Box} x \: \& \: v R_{\Box} y$. $\mathcal{M}, u \models \Box p$, $\mathcal{M}, v \models \Box q$ implies that for each $x \in R_{\Box}(u) \: \mathcal{M}, x \models p$ and for each $y \in R_{\Box}(v)$. $R x y w'$ implies $\mathcal{M}, w' \models p \bullet q$. Thus, $\mathcal{M}, w \models \Box(p \bullet q)$
\end{proof}

One may extend the notion of a bounded morphism for the relevant case to have homomorphisms between Kripke frames and models that preserve truth.

\begin{definition}
  Let $\mathcal{F}_1, \mathcal{F}_2$ be Kripke frames, a map $f : \mathcal{F}_1 \to \mathcal{F}_2$ is a bounded morphism, if:

  \begin{enumerate}
    \item $R_1 u v w \Rightarrow R_2 f(u) f(v) f(w)$
    \item $R_2 f(u) v' w' \Rightarrow \exists v, w \: v' \leq f(v) \: \& \: f(w) \leq w' \: \& \: R u v w$
    \item $R_2 u' f(v) w' \Rightarrow \exists u, w \: u' \leq f(u) \: \& \: f(w) \leq w' \: \& \: R u v w$
    \item $R_2 u' v' f(w) \Rightarrow \exists u, v \: u' \leq f(u) \: \& \: v' \leq f(v) \: \& \: R u v w$
    \item $f[(R_{\nabla})_1(x)] = (R_{\nabla})_2 (f(x))$, $\nabla = \Box, \Diamond$
    \item $f^{-1}[\mathcal{O}_2] = \mathcal{O}_1$
  \end{enumerate}
\end{definition}
By the notation $\mathcal{F}_1 \twoheadrightarrow \mathcal{F}_2$ we mean that there exists a surjective bounded morphism $f : \mathcal{F}_1 \twoheadrightarrow \mathcal{F}_2$ and we call such a map \emph{$p$-morphism}. Let $\mathcal{F}_1$, $\mathcal{F}_2$ be Kripke frames, $\vartheta_1$, $\vartheta_2$ valuations on $\mathcal{F}_1$ and $\mathcal{F}_2$ correspondingly, and $f : \mathcal{F}_1 \to \mathcal{F}_2$ a bounded morphism. Then $f : \mathcal{M}_1 \to \mathcal{M}_2$ is a bounded morphism of models, if
\begin{center}
  $\mathcal{M}_1, w \models p_i \Leftrightarrow \mathcal{M}_2, f(w) \models p_i$ for each propositional variable $p_i$.
\end{center}

The following lemma has the standard proof given in \cite{blackburn_rijke_venema_2001} \cite{chagrov1997modal}.

\begin{lemma}
$ $

  \begin{enumerate}
    \item $\mathcal{M}_1, w \models \varphi \Leftrightarrow \mathcal{M}_2, f(w) \models \varphi$.
    \item $\mathcal{F}_1 \twoheadrightarrow \mathcal{F}_2$ implies $\operatorname{Log}(\mathcal{F}_1) \subseteq \operatorname{Log}(\mathcal{F}_2)$.
    \item $\mathcal{F}_1 \cong \mathcal{F}_2$ implies $\operatorname{Log}(\mathcal{F}_1) = \operatorname{Log}(\mathcal{F}_2)$.
  \end{enumerate}
\end{lemma}

\section{Residuated distributive modal algerbas}

In this section, we study algebraic semantics and canonical extensions for substructural distributive modal logic. The required lattice-theoretic and canonical extensions definitions and notations are explained in the appendix. Let us define a residuated lattice \cite{Jipsen}.

\begin{definition}
  A residuated lattice is an algebra $\mathcal{R} = \langle \mathcal{L}, \cdot, \setminus, /, \varepsilon \rangle$, where $\mathcal{L}$ is a bounded lattice, $\cdot$ is a binary associative monotone operation,
  $\varepsilon$ is a multiplicative identity, $\setminus$ and $/$ are residuals, that is, the following equivalence holds for all $a, b, c \in \mathcal{L}$:

  \begin{center}
    $b \leq a \setminus c \Leftrightarrow a \cdot b \leq c \Leftrightarrow a \leq c / b$
  \end{center}
\end{definition}

Note that the class of all residuated lattices forms a variety since the quasi-identities above might be equivalently reformulated as identities, see \cite[Lemma 2.3]{Jipsen}.

A residuated lattice is called bounded distributive if its lattice reduct is a bounded distributive lattice. A residuated lattice morphism is a map $f : \mathcal{L}_1 \to \mathcal{L}_2$ that commutes with all operations in a usual way.

Let us recall the essential facts about (prime) filters on bounded distributive residuated lattices, see \cite{galatos2003varieties} \cite{Urq}. As a matter of fact, these statements hold for an arbitrary distributive lattice ordered semigroup since those properties of filters and their products do not depend on residuals.

\begin{lemma}\label{filter}
  Let $\mathcal{L}$ be a bounded distributive residuated lattice. Let $X, Y \subseteq \mathcal{L}$ and
  $X \cdot Y = \{ x \cdot y \: | \: x \in X, y \in Y \}$. Let us define $X \bullet Y = \uparrow X \cdot Y$, then
  \begin{enumerate}
    \item If $Z \subseteq \mathcal{L}$ is a filter, then $X \cdot Y \subseteq Z$ iff $X \bullet Y \subseteq Z$.
    \item If $X, Y \subseteq \mathcal{L}$ are filters, then $X \bullet Y$ is a filter.
    \item Let $X, Y, Z$ be filters in $\mathcal{L}$ and $Z$ is prime such that $XY \subseteq Z$, then there exist prime filters $X', Y'$ such that $X \subseteq X'$, $Y \subseteq Y'$, and $X' Y' \subseteq Z$.
  \end{enumerate}
\end{lemma}

A residuated distributive modal algebra is a bounded distributive residuated lattice extended with the operators $\Box$ and $\Diamond$ that distribute over finite infima and suprema respectively. One may also consider such algebras as full Lambek algebras \cite{ono1993semantics} \cite{Ono2019} reducts of which are bounded distributive lattices. Here, modalities are merely the ${\bf K}$-like operators without any additional requirements except the connection between $\Box$ and $\cdot$. We require that $\Box$ is also ``weakly normal'' with respect to the product. Such a ``weak normality'' corresponds to the promotion principle which is widespread in linear logic. This principle often has the ${\bf K}4$ form, where formulas in the premise are boxed from the left. This version of the promotion rule is rather the ${\bf K}$-rule than the ${\bf K}4$ one:

\begin{prooftree}
\AxiomC{$\varphi_1 \bullet \dots \bullet \varphi_n \vdash \psi$}
\UnaryInfC{$\Box \varphi_1 \bullet \dots \bullet \Box \varphi_n \vdash \Box \psi$}
\end{prooftree}

The inference rule also allows one to obtain the Kripke axioms formulated in terms of residuals as, e.g., $\Box(\varphi \setminus \psi) \Rightarrow \Box \varphi \setminus \Box \psi$ and $\Box (\psi / \varphi) \Rightarrow \Box \psi / \Box \varphi$.

This ``weak normality'' requirement is introduced as the additional inequation, more precisely:
\begin{definition}
  A residuated distributive modal algebra (RDMA) is an algebra $\mathcal{M} = \langle \mathcal{R}, \Box, \Diamond \rangle$ with the following conditions for each $a, b \in \mathcal{R}$:

  \begin{enumerate}
    \item $\Box (a \land b) = \Box a \land \Box b$, $\Box \top = \top$
    \item $\Diamond (a \lor b) = \Diamond a \lor \Diamond b$, $\Diamond \bot = \bot$
    \item $\Box a \cdot \Box b \leq \Box (a \cdot b)$
  \end{enumerate}

  An RDMA homomorphism is a bounded distributive residuated lattice homomorphism $f : \mathcal{M}_1 \to \mathcal{M}_2$ such that $f(\Box a) = \Box (f(a))$ and $f(\Diamond a) = \Diamond (f (a))$.
\end{definition}

One may associate with an arbitrary substructural normal modal logic its variety of RDMAs as follows:

\begin{definition}
  Let $\Lambda$ be a substructural normal modal logic, $\mathcal{V}_{\Lambda}$ is a variety defined by the set of inequations $\{ \varphi \leq \psi \: | \: \Lambda \vdash \varphi \vdash \psi \}$.
\end{definition}
Note that $\varphi$, $\psi$ are terms of the signature $\langle \wedge, \vee, \bot, \top, \cdot, \setminus, /, \varepsilon, \Box, \Diamond \rangle$ in such inequations as $\varphi \leq \psi$. One has an algebraic completeness for each substructural distributive normal modal logic as usual.

\begin{theorem}\label{linden}
  Let $\Lambda$ be a substructural normal modal logic, then there exists an RDMA $\mathcal{R}_{\Lambda}$ such that $\varphi \vdash \psi \in \Lambda$ iff $\mathcal{R}_{\Lambda} \models \varphi \leq \psi$.
\end{theorem}
One such RDMA is the free countably generated algebra in the variety $\mathcal{V}_{\Lambda}$, the Lindenbaum-Tarski algebra up to isomorphism.

The following statement also holds according to the general technique, see \cite{GOLDBLATT1989173}:

\begin{lemma}
  Let $\mathcal{\Lambda}$ be a residuated distributive normal modal logic, then the map $\Lambda \mapsto \mathcal{V}_{\Lambda}$ is the isomorphism between the lattice of residuated distributive normal modal logics and the lattice of varieties of RDMAs.
\end{lemma}

We define a completely distributive residuated perfect lattice as a distributive version of a residuated perfect one defined in \cite{dunn2005}.

\begin{definition}
  A distributive residuated lattice
  $\mathcal{L} = \langle L, \bigvee, \bigwedge, \cdot, \setminus, /, \varepsilon \rangle$ is called perfect distributive residuated lattice, if:

  \begin{itemize}
    \item Its lattice reduct is perfect distributive, see Definition~\ref{cdlattice}.
    \item $\cdot$, $\setminus$, and $/$ are binary operations on $L$ such that $/$ and $\setminus$ right and left residuals of $\cdot$, repsectively; $\cdot$ is a complete operator on $\mathcal{L}$, and $/ : \mathcal{L} \times \mathcal{L}^{\delta} \to \mathcal{L}$, $\setminus : \mathcal{L}^{\delta} \times \mathcal{L} \to \mathcal{L}$ are complete dual operators, where $\mathcal{L}^{\delta}$ is the dual of $\mathcal{L}$.
  \end{itemize}
\end{definition}

Here we formulate canonical extensions for bounded distributive lattices with a residuated family in the fashion of \cite{Gehrke2014}, where the author introduces canonical extensions for Heyting algebras. Here we take a generalised version of that construction formulated for residuated lattices, see \cite{gehrketopological}.

\begin{lemma}
  Let $\mathcal{L} = \langle L, \cdot, \setminus, /, \varepsilon \rangle$ be a bounded distributive residuated lattice, then $\mathcal{L}^{\sigma} = \langle L^{\sigma}, \cdot^{\sigma}, \setminus^{\pi}, /^{\pi}, \varepsilon \rangle$ is a perfect distributive residuated lattice.
\end{lemma}

Instead of proof that repeats this one \cite{gehrketopological}, we just define $\cdot^{\sigma}$, $\setminus^{\pi}$, and $/^{\pi}$ explicitly. Here we note that the canonical extension of a lattice reduct is a perfect distributive lattice \cite{gehrke1994bounded}.

Let $a, a' \in \mathcal{F}(\mathcal{L}^{\sigma})$ and $b \in \mathcal{F}(\mathcal{L}^{\sigma})$, then

\begin{enumerate}
  \item $a \setminus^{\pi} b = \bigvee \{ x \setminus y \: | \: a \leq x \in \mathcal{L} \ni y \leq b\}$ and similarly for the right residual
  \item $a \cdot^{\sigma} a' = \bigwedge \{ x \cdot x' \: | \: a \leq x \in \mathcal{L} \: \& \: a \leq x' \in \mathcal{L} \}$
\end{enumerate}

Let $a, b \in \mathcal{L}^{\sigma}$, then.
\begin{enumerate}
  \item $a \cdot^{\sigma} b =
  \bigvee \{ x \cdot^{\sigma} y \: | \: a \geq x \in \mathcal{F}(\mathcal{L}^{\sigma}) \: \& \: b \geq y \in \mathcal{F}(\mathcal{L}^{\sigma}) \}$
  \item $a \setminus^{\pi} b = \bigwedge \{ x \setminus^{\pi} y \: | \: a \geq x \in \mathcal{F}(\mathcal{L}^{\sigma}) \: \& \: b \leq y \in \mathcal{I}(\mathcal{L}^{\sigma})\}$ and $b /^{\pi} a$ is defined similarly
\end{enumerate}
The residuation property follows from the meet-density of $\mathcal{F}(\mathcal{L}^{\sigma})$ and join-density of $\mathcal{I}(\mathcal{L}^{\sigma})$ in $\mathcal{L}^{\sigma}$. Hence $\mathcal{L}^{\sigma}$ is a perfect distributive residuated lattice.

Let us describe the discrete duality for perfect distributive residuated lattices. Here we concretise the construction that establishes the discrete duality between perfect residuated lattices and perfect posets with ternary relation in \cite{dunn2005} within a distributive setting. We piggyback the Raney representation of perfect distributive lattices as algebras of downsets of completely join-irreducible elements \cite{raney1952completely} that generalise Birkhoff representation for finite lattices \cite{birkhoff1937rings}. We just recall that any perfect distributive lattice $\mathcal{L}$ is isomorphic to the lattice $\operatorname{Down}(\mathcal{J}^{\infty}(\mathcal{L}))$ mapping $a \in \mathcal{L}$ to $\mathcal{J}^{\infty}(\mathcal{L}) \cap \downarrow a$.

This representation might be extended to the duality between the categories of perfect distributive lattices and posets.

Let $\mathcal{L}$ be a perfect distributive residuated lattice. We define the relation $R \subseteq \mathcal{J}^{\infty}(\mathcal{L}) \times \mathcal{J}^{\infty}(\mathcal{L}) \times \mathcal{J}^{\infty}(\mathcal{L})$ as $R a b c \Leftrightarrow a \cdot b \leq c$.
Let us put $\mathcal{O} = \uparrow \varepsilon$, where $\varepsilon$ is a multiplicative identity. The structure $\mathcal{L}_{+} = \langle \mathcal{J}^{\infty}(\mathcal{L}), \leq, R, \mathcal{O} \rangle$ is the \emph{dual frame} of a perfect distributive residuated lattice $\mathcal{L}$.

Let $\langle W, \leq \rangle$ be a poset and $R \subseteq W^{3}$, $\mathcal{O}$ with the conditions (ii)-(vi) from Definition~\ref{frame}. Let us define the following operations on $\operatorname{Up}(W, \leq)$:
\begin{itemize}
  \item $A \setminus B = \{ w \in W \: | \: \forall u, v \in W \: R u w v \: \& \: u \in A \Rightarrow v \in B \}$
  \item $B / A = \{ w \in W \: | \: \forall u, v \in W \: R w u v \: v \in A \Rightarrow v \in B \}$
  \item $A \cdot B = \{ w \in W \: | \: \exists u, v \in W \: R u v w \: \& \: u \in A \: \& \: v \in B \}$
\end{itemize}

These operations are clearly well defined. Let us check $A \setminus B \in \operatorname{Up}(W, \leq)$, if $A, B \in \operatorname{Up}(W, \leq)$. Let $x \in A \setminus B$ and $x \leq y$, then $A \cdot \{ x \}\subseteq B$. Let $w \in \{ y \} \cdot A$, then there exists $a \in A$ such that $R y a w$, then $R x a w$ by the item (iii), Definition~\ref{frame}. Thus, $\{ y \} \cdot A \subseteq B$. Similarly, $B / A \in \operatorname{Up}(W, \leq)$. The residuation property for these operations holds immediately.

The following theorem establishes the discrete duality between perfect distributive residuated algebras and posets with a ternary relation that encodes the product. Let us call such a poset with the relation as above a \emph{ternary Kripke frame}.

\begin{theorem}\label{restheorem}
  $ $

  \begin{enumerate}
  \item Let $\mathcal{R}$ be a perfect distributive residuated lattice, then $\mathcal{R} \cong (\mathcal{R}_{+})^{+}$.
  \item Let $\mathcal{F}$ be a ternary Kripke frame, then $\mathcal{F} \cong (\mathcal{F}^{+})_{+}$.
\end{enumerate}
\end{theorem}

\begin{proof}
$ $

  \begin{enumerate}
    \item Let $R$ be a lattice reduct of $\mathcal{R}$. Accoding to the Raney representation, $\eta : R \cong \operatorname{Up}(\mathcal{J}^{\infty}(R), \leq_{\delta})$ such that $\eta : a \mapsto \{ b \in \mathcal{J}^{\infty}(R) \: | \: a \leq_{\delta} b\}$, where $\leq_{\delta}$ is a dual order on $\mathcal{J}^{\infty}(R)$. Let us ensure that this isomorphism also preserves products and residuals.

    Let $z \in \eta(a) \cdot \eta(b)$, then there exists $x \in \eta(a)$ and $y \in \eta(b)$ such that $x \cdot y \leq_{\delta} z$. That is, $a \leq_{\delta} x$ and $b \leq_{\delta} y$, so $a \cdot b \leq_{\delta} z$, then $z \in \eta(a \cdot b)$.

    Let $z \in \eta(a \cdot b)$, then $z$ is a join-irreducible element such that $a \cdot b \leq_{\delta} z$. Then $\uparrow z$ is a prime filter since $R$ is perfect distributive. It is clear that $\uparrow a \cdot \uparrow b \subseteq \uparrow b$. By Lemma~\ref{filter}, there exists prime filters $A$ and $B$ such that $\uparrow a \subseteq A$, $\uparrow b \subseteq B$,
    $A \cdot \uparrow b \subseteq \uparrow z$ and $\uparrow a \cdot B \subseteq \uparrow z$. $R$ is completely distributive, so there exists $a', b' \in \mathcal{J}^{\infty}(R)$ such that $\uparrow a' = A$ and $\uparrow b' = B$.
    Moreover, $\uparrow a \subseteq \uparrow a'$ and $\uparrow b \subseteq \uparrow b'$ implies $a \leq_{\delta} a'$ and $b \leq_{\delta} b'$, so $a' \in \eta(a')$ and $b' \in \eta(b)$. So $a' \cdot b' \leq_{\delta} z$, and, thus, $z \in \eta(a) \cdot \eta(b)$.

    $\eta$ preserves left and right residuals similarly to \cite[Lemma 6.10]{galatos2003varieties}. That was shown for arbitrary bounded residuated residuated lattices and the extended Priestley embedding.
    \item $\varepsilon : \langle W, \leq \rangle \to \langle \mathcal{J}^{\infty}(W, \leq), \leq_{\delta} \rangle$ is a poset isomorphism such that $\varepsilon : a \mapsto \uparrow a$. This isomorphism might be extended to the frame isomorphism via the frame conditions that connect a ternary relation with the partial order.
  \end{enumerate}
\end{proof}

\section{Discrete duality and completeness}

In this section, we establish a discrete duality between the categories of all Kripke frames and the category of all perfect residuated distributive modal algebras. We show that the Thomason theorem \cite{thomason1975categories} holds for normal residuated distributive modal logics.

\begin{definition}
  Let $\mathcal{L}$ be a perfect distributive residuated lattice and $\Box, \Diamond$ unary operators on $\mathcal{L}$, then
  $\mathcal{M} = \langle \mathcal{L}, \Box, \Diamond \rangle$ is called a perfect distributive residuated modal algebra, if for each where $A \subseteq \mathcal{L}$ and $a, b \in \mathcal{L}$
  \begin{itemize}
    \item $\Box \bigwedge A = \bigwedge \{ \Box a \: | \: a \in A \}$
    \item $\Diamond \bigvee A = \bigvee \{ \Diamond a \: | \: a \in A\}$
    \item $\Box a \cdot \Box b \leq \Box (a \cdot b)$
  \end{itemize}
\end{definition}
Given $\mathcal{M}$, $\mathcal{N}$ perfect residuated distributive modal algebras, a map $\mathcal{M} \to \mathcal{N}$ is a homomorphism if $f$ is a complete lattice homomorphism that preserves product, residuals, modal operators, and the multiplicative identity.

Let us show that the variety of all RDMAs is closed under canonical extensions.

\begin{lemma}\label{pRDMA}
  Let $\mathcal{R}$ be a istributive residuated lattice and $\mathcal{M} = \langle \mathcal{R}, \Box, \Diamond \rangle$ an RDMA, then $\mathcal{M}^{\sigma} = \langle \mathcal{R}^{\sigma}, \Box^{\pi}, \Diamond^{\pi} \rangle = \langle \mathcal{R}^{\sigma}, \Box^{\sigma}, \Diamond^{\sigma} \rangle$ is a perfect DRMA.
\end{lemma}

\begin{proof} The lattice reduct of $\mathcal{R}^{\sigma}$ is a perfect distributive lattice, \cite{gehrke1994bounded}.
In fact, one needs to show that the inequation $\Box a \cdot \Box b \leq \Box (a \cdot b)$ is canonical. Firstly, let us suppose that $a, b \in \mathcal{C}(\mathcal{M}^{\sigma})$. Note that $\Box^{\sigma} a \cdot^{\sigma} \Box^{\sigma} b = \bigwedge \{ \Box x \cdot \Box y \: | \: a \leq x \in \mathcal{M}, b \leq y \in \mathcal{M} \}$ that follows from the definition of a filter element, the fact that $\Box^{\sigma}$ preserves all infina and $\cdot^{\sigma}$ is an order-preserving operation. Then:

  $\begin{array}{lll}
  & \Box^{\sigma} a \cdot^{\sigma} \Box^{\sigma} b = & \\
  & \bigwedge \{ \Box x \cdot \Box y \: | \: a \leq x \in \mathcal{L}, \leq x \in \mathcal{L} \} \leq \bigwedge \{ \Box (x \cdot y) \: | \: a \leq x \in \mathcal{L} \: \& \: b \leq x \in \mathcal{L} \} = & \\
  & \bigwedge \Box^{\sigma} \{(x \cdot y) \: | \: a \leq x \in \mathcal{L} \: \& \: b \leq x \in \mathcal{L} \} = & \\
  & \Box^{\sigma} \bigwedge \{(x \cdot y) \: | \: a \leq x \in \mathcal{L} \: \& \: b \leq x \in \mathcal{L} \} = \Box^{\sigma} (a \cdot^{\sigma} b)&
  \end{array}$

  \vspace{\baselineskip}

  Let $a, b \in \mathcal{L}^{\sigma}$, then

  $\begin{array}{lll}
  & \Box^{\sigma} a \cdot^{\sigma} \Box b = \bigvee \{ \Box^{\sigma} x \cdot^{\sigma} \Box^{\sigma} y \: | \: a \geq x \in \mathcal{C}(\mathcal{L}^{\sigma}) \: \& \: b \geq y \in \mathcal{C}(\mathcal{L}^{\sigma}) \} \leq & \\
  & \bigvee \{ \Box^{\sigma} (x \cdot^{\sigma} y) \: | \: a \geq x \in \mathcal{C}(\mathcal{L}^{\sigma}) \: \& \: b \geq y \in \mathcal{C}(\mathcal{L}^{\sigma}) \} \leq & \\
  & \Box^{\sigma} \bigvee \{ x \cdot^{\sigma} y \: | \: a \geq x \in \mathcal{C}(\mathcal{L}^{\sigma}) \: \& \: b \geq y \in \mathcal{C}(\mathcal{L}^{\sigma}) \} = \Box^{\sigma} (a \cdot b)&
  \end{array}$
\end{proof}

\begin{definition}
  A substructural normal modal logic $\mathcal{L}$ is called canonical, if $\mathcal{V}_{\mathcal{L}}$ is closed under canonical extensions
\end{definition}

The complex algebra of a Kripke frame $\mathcal{F} = \langle W, \leq, R, R_{\Box}, R_{\Diamond}, \mathcal{O} \rangle $ is the complex algebra of the underlying residuated frame $\mathcal{F}^{+}$ with the modal operators defined as $[R_{\Box}] A = \{ u \in W \: | \: \forall w \: (u R_{\Box} w \Rightarrow w \in A)\}$ and $\langle R_{\Diamond} \rangle = \{ u \in W \: | \: \exists w \: (u R_{\Diamond w} \: \& \: w \in A)\}$. Here $A$ is upwardly closed subset. These operations are well-defined. The dual frame of a perfect RDMA $\mathcal{M} = \langle M, \bigvee, \bigwedge, \Box, \Diamond, \cdot, \setminus, /, \varepsilon \rangle$ is the dual frame $\mathcal{M}_{+}$ of an underlying perfect distributive residuated lattice with binary relations on completely join irreducible elements introduced as $a R_{\Box} b \Leftrightarrow \Box \kappa(a) \leq \kappa(b)$ and $a R_{\Diamond} b \Leftrightarrow a \leq \Diamond b$. Here, $\kappa$ is an order isomorphism between between $\mathcal{J}^{\infty}(\mathcal{M})$ and $\mathcal{M}^{\infty}(\mathcal{M})$.

Logically, Kripke frames and their complex algebras are connected with each other as follows:
\begin{proposition}\label{complex}
  Let $\mathcal{F}$ be a Kripke frame, then $\operatorname{Log}(\mathcal{F}) = \operatorname{Log}(\mathcal{F}^{+}) = \{ \varphi \vdash \psi \: | \: \mathcal{F}^{+} \models \varphi \leq \psi \}$
\end{proposition}

The following discrete duality theorem is merely a combination of Theorem~\ref{restheorem} and the similar fact proved for distributive modal algebras and frames for distributive modal logics \cite{Distr}.

\begin{theorem} \label{discrete}
$ $

  \begin{enumerate}
    \item Let $\mathcal{F}$ be a Kripke frame, then $\mathcal{F} \cong (\mathcal{F}^{+})_{+}$
    \item Let $\mathcal{M}$ be a perfect DRMA, then $\mathcal{M} \cong (\mathcal{M}_{+})^{+}$
    \item Functors $(.)_{+}: \operatorname{pDRMA} \rightleftarrows \operatorname{KF} : (.)^{+}$ establish a dual equivalence between the categories of all Kripke frames and all perfect RDMAs.
  \end{enumerate}
\end{theorem}

\begin{proof} It is easy to check that if $f : \mathcal{F}_1 \to \mathcal{F}_2$ is a bounded morphism of Kripke frames, then $f^{+} : {\mathcal{F}_2}^{+} \to {\mathcal{F}_1}^{+}$ such that $f^{+} : A \mapsto f^{-1}[A]$ is a perfect DRMA morphism. It is immediate that $h_{+} : {\mathcal{M}_2}_{+} \to {\mathcal{M}_1}_{+}$ is a bounded morphism, where $h : \mathcal{M}_1 \to \mathcal{M}_2$ is a perfect DRMA morphism. Thus, the dual equivalence follows from the previous two items and the lemma that claims that $(.)_{+}$ and $(.)^{+}$ are contravariant functors.
\end{proof}

The discrete duality established above together with canonical extensions of residuated distributive modal algebras provides the following consequence:

\begin{theorem}\label{complete}
  Let $\mathcal{L}$ be a canonical substructural distributive modal logic, then $\mathcal{L}$ is Kripke complete.
\end{theorem}

\begin{proof}
  The proof is similar to the analogous fact proved in \cite{Distr}, but we reproduce a sketch.

  Let $\varphi \vdash \psi \in \mathcal{L}$. Then $\mathcal{M}_{\mathcal{L}} \models \varphi \leq \psi$ by Theorem~\ref{linden}.
  But $\mathcal{M}_{\mathcal{L}} \models \varphi \leq \psi$ iff and only $(\mathcal{M}_{\mathcal{L}})^{\sigma} \models \varphi \leq \psi$ since $\mathcal{M}_{\mathcal{L}}$
  is a subalgebra of $(\mathcal{M}_{\mathcal{L}})^{\sigma}$ by Lemma~\ref{pRDMA} and the condition according to which $\mathcal{V}_{\mathcal{L}}$ is a canonical variety.
  By discrete duality, Theorem~\ref{discrete}, $({(\mathcal{M}_{\mathcal{L}})^{\sigma}}_{+})^{+} \models \varphi \leq \psi$.
  By Proposition~\ref{complex}, ${(\mathcal{M}_{\mathcal{L}})^{\sigma}}_{+} \models \varphi \vdash \psi$. In fact, ${(\mathcal{M}_{\mathcal{L}})^{\sigma}}_{+}$ is a canonical frame of a given logic $\mathcal{L}$ and we showed that ${(\mathcal{M}_{\mathcal{L}})^{\sigma}}_{+} \models \mathcal{L}$.
\end{proof}

As a consequence, the minimal normal substructural distributive modal logic is complete with respect to the class of all Kripke frames.
\begin{corollary}\label{completeK}
  ${\bf L}_{\bf K}$ is Kripke-complete.
\end{corollary}
Now we show that the following sequents that describe modalities as storage operators are canonical ones. This lemma partially repeats here \cite[Propositions 6.7 -- 6.10]{dunn2005}.
\begin{lemma}\label{canonseq} The following sequents are canonical:
  \begin{enumerate}
    \item $\Box p \bullet q \dashv \vdash q \bullet \Box p$
    \item $\Box p \vdash \Box p \bullet \Box p$
    \item $\Box p \bullet q \vdash \Box p \bullet q \bullet \Box p$, $q \bullet \Box p \vdash \Box p \bullet q \bullet \Box p$
    \item $\Box p \vdash \varepsilon$
    \item $\Box (p \land q) \dashv \vdash \Box p \bullet \Box q$
    \item $\Box (p \setminus q) \vdash \Diamond p \setminus \Diamond q$ and $\Box (p / q) \vdash \Diamond p / \Diamond q$
  \end{enumerate}
\end{lemma}

\begin{proof}
Let us check only the third sequent. The rest sequents might be checked similarly.

Let ${\bf K}_{nc}$ be ${\bf L}_{{\bf K}} \oplus \Box a \bullet b \vdash \Box a \bullet b \bullet \Box a$ and $\mathcal{V}_{{\bf K}_{nc}}$ its variety. Let us show that $\mathcal{V}_{{\bf K}_{nc}}$ is closed under canonical extensions.

Let $\mathcal{M} \in \mathcal{V}_{{\bf K}_{nc}}$, one needs to check that $\mathcal{M}^{\sigma} \in \mathcal{V}_{{\bf K}_{nc}}$. Let $a, b \in \mathcal{C}(\mathcal{M}^{\sigma})$.

  $\begin{array}{lll}
  & \Box^{\sigma} a \cdot^{\sigma} b = \bigwedge \{ \Box x \: | \: a \leq x \in \mathcal{M} \} \cdot^{\sigma} \bigwedge \{ y \: | \: b \leq y \in \mathcal{M} \} = & \\
  & \bigwedge \{ \Box x \cdot y \: | \: a \leq x \in \mathcal{M} \: \& \: b \leq y \in \mathcal{M} \} \leq & \\
  & \bigwedge \{ \Box x \cdot y \cdot \Box x \: | \: a \leq x \in \mathcal{M} \: \& \: b \leq y \in \mathcal{M} \} = \Box^{\sigma} a \cdot^{\sigma} b \cdot^{\sigma} \Box^{\sigma} a&
  \end{array}$

  \vspace{\baselineskip}

  Let $a, b \in \mathcal{M}^{\sigma}$, then

  $\begin{array}{lll}
  &\Box^{\sigma} a \cdot^{\sigma} b = \bigvee \{ \Box^{\sigma} x \cdot^{\sigma} y \: | \: a \geq x \in \mathcal{C}(\mathcal{M}^{\sigma}) \: \& \: b \geq y \in \mathcal{C}(\mathcal{M}^{\sigma}) \} \leq & \\
  & \bigvee \{ \Box^{\sigma} x \cdot^{\sigma} y \cdot^{\sigma} \Box^{\sigma} x \: | \: a \geq x \in \mathcal{C}(\mathcal{M}^{\sigma}) \: \& \: b \geq y \in \mathcal{C}(\mathcal{M}^{\sigma}) \} = \Box^{\sigma} a \cdot^{\sigma} b \cdot^{\sigma} \Box^{\sigma} a&
  \end{array}$
\end{proof}

The following completeness theorems follows from Theorem~\ref{complete}, Corollary~\ref{completeK}, and Lemma~\ref{canonseq}.

\begin{corollary}\label{structcompelete}
  Let $\Gamma$ be the set of all sequents from the lemma above and $\Delta \subseteq \Gamma$. Then $\mathcal{L} = {\bf L}_{{\bf K}} \oplus \Delta$ is Kripke complete.
\end{corollary}

One may also consider the subexponential polymodal case introduced in \cite{kanovich2019subexponentials}. Let us define a subexponential signature:

\begin{definition} A subexponential signature is an ordered quintuple: $\Sigma = \langle \mathcal{I}, \preceq, \mathcal{W}, \mathcal{C}, \mathcal{E} \rangle$,
where $\langle \mathcal{I}, \preceq \rangle$ is a preorder. $\mathcal{W}, \mathcal{C}, \mathcal{E}$ are upwardly closed subsets of $\mathcal{I}$ and $\mathcal{W} \cap \mathcal{C} \subseteq \mathcal{E}$.
\end{definition}

Let us define the following axioms:

\begin{itemize}
  \item $\Box_{s_1} p \bullet \Box_{s_2} q \vdash \Box_{s} (p \bullet q)$, $s \preceq s_1, s_2$
  \item $\Box_s p \vdash p$
  \item $\Box_s p \vdash \Box_s \Box_s p$
  \item $\Box_s p \bullet q \vdash \Box_s p \bullet q \bullet \Box_s p$, $q \bullet \Box_s p \vdash \Box_s p \bullet q \bullet \Box_s p$, where $s \in \mathcal{C}$
  \item $\Box_s p \bullet q \dashv \vdash p \bullet \Box_s q$, $s \in \mathcal{E}$
  \item $\Box_s p \vdash {\bf 1}$, where $s \in \mathcal{W}$
\end{itemize}

The system $\operatorname{DSMALC}_{\Sigma}$ is a substructural distributive polymodal logic with modal axioms as above plus $\Box_s (p \land q) \dashv \vdash \Box_s p \land \Box_s q$ and $\Box_s \top \dashv \vdash \top$ for each $s \in \mathcal{\Sigma}$. For simplicity, let us put the diamond axioms $\Diamond_s (p \lor q) \Leftrightarrow \Diamond_s p \lor \Diamond_s q$ and $\Diamond_s \bot \dashv \vdash \bot$ with the ${\bf S}4$ axioms for $\Diamond$ $p \vdash \Diamond_s p$ and $\Diamond_s \Diamond_s p \vdash \Diamond_s p$ for each $s \in \mathcal{\Sigma}$ axioms without any additional postulates. The modal inference rules have the form: from $\varphi \vdash \psi$ infer $\nabla_{s_2} \varphi \vdash \nabla_{s_1} \psi$ for $s_1 \preceq s_2$, where $nabla = \Box, \Diamond$.

\begin{theorem}
   $\operatorname{DSMALC}_{\Sigma}$ is canonical and, thus, Kripke complete.
\end{theorem}

\begin{proof}
  One may show that the variety $\mathcal{V}_{\operatorname{DSMALC}_{\Sigma}}$ is canonical showing that
  $\Box_{s_1} p \cdot \Box_{s_2} q \leq\Box_{s} (p \cdot q)$ for $s \preceq s_1, s_2$ is a canonical inequation similarly to Lemma~\ref{pRDMA}. The whole statement follows from the observation above and Theorem~\ref{complete}, Lemma~\ref{canonseq}, and Corollary~\ref{structcompelete}.
\end{proof}

\section{Topological duality}

In this section, we characterise a topological duality for residuated distributive modal algebras
in the same fashion as in
\cite{Esakia2019} \cite{Sambin1988TopologyAD}. We consider topological Kripke frames, ternary Kripke frames defined on Priestley spaces with binary modal relations, the category of which is dually equivalent to the category of all RDMAs. Alternatively, one may characterise such a duality in terms of general descriptive frames following the Goldblatt's approach \cite{goldblatt1993mathematics}. See the Appendix to have an explanation of the Priestley duality related definitions, terms, and notations.

Firstly, we consider a Priestley-style duality for residuated distributive bounded lattices. We piggyback the construction obtained by Galatos in his PhD thesis \cite{galatos2003varieties}. This construction is a noncommutative generalisation of relevant spaces, the dual spaces of relevant algebras studied by Urquhart \cite{Urq}. In fact, those spaces and their extensions with modal relations are the instances of relational Priestley spaces \cite{GOLDBLATT1989173}.

\begin{definition}
  Let $\mathcal{X} = \langle X, \tau, \leq \rangle$ be a Priestley space, $R \subseteq X^3$ and $E \subseteq X$.
  A bDRL-space is a tuple $\mathcal{X} = \langle X, \tau, \leq, R, E \rangle$ such that:

  \begin{enumerate}
    \item For all $x, y, z, w \in X$ there exists $u \in X$  such that $R(x,y,u)$ and $R(u,z,w)$ iff
    there exists $v \in X$ such that $R(y,z,v)$ and $R(x, v, w)$.
    \item For all $x, y, z, w \in X$ if $x \leq y$ then $R(y, u, v)$ implies $R(x, u, v)$,
    $R(u, y, v)$ implies $R(u, x, v)$, and $R(u, v, y)$ implies $R(u, v, x))$.
    \item Let $A, B \subseteq X$ be upwardly closed clopens then $R[A, B, \rule{0.25cm}{0.15mm} ]$, $\{ z \in X \: | \: R[z, B, \rule{0.25cm}{0.15mm}] \subseteq A\}$, and $\{ z \in X \: | \: R[B, z, \rule{0.25cm}{0.15mm} ]\}$ are also clopens.
    \item For all $x, y, z \in X$  $\neg R(x,y,z)$ iff there exists upwardly closed clopens $A, B \subseteq X$ such that that there exist $x \in A$ and $y \in B$ such that $z \notin R[A,B,\rule{0.25cm}{0.15mm}])$.
    \item $E$ is upwardly closed cloped such that if $A$ is clopen then $R[E, A, \rule{0.25cm}{0.15mm}] = R[A, E, \rule{0.25cm}{0.15mm}] = A$.
  \end{enumerate}

  Here $R[A, B, \rule{0.25cm}{0.15mm} ]$ denotes $\{ c \in X \: | \: \exists a \in A \: \exists b \in B \:\: R a b c \}$.
\end{definition}

We note that such a space is totally disconnected concerning a ternary relation according to the fourth condition. We introduce topological Kripke frames as bDRL-spaces with modal relations as follows.

\begin{definition} A \emph{modal bDRL-space} is a structure $\mathcal{X} = \langle X, \tau, \leq, R, R_{\Box}, R_{\Diamond}, E \rangle$, where $\langle X, \tau, \leq, R, E \rangle$ is a bDRL space and the following conditions hold:

\begin{enumerate}
\item If $A$ is upwardly closed clopen, then $R_{\Box}(A), R_{\Diamond}(A)$ are upwardly closed clopens.
\item For each $x \in X$, $R_{\Box}(x)$ and $R_{\Diamond}(x)$ are closed.
\item $R u v w \: \& \: w R_{\Box} w' \Rightarrow \exists x, y \in W \:\: R x y w' \: \& \: u R_{\Box} x \: \& \: v R_{\Box} y$.
\item $u \leq v \: \& \: v R_{\Box} w \Rightarrow u R_{\Box} w$.
\item $u \leq v \: \& \: u R_{\Diamond} w \Rightarrow v R_{\Diamond} w$.
\end{enumerate}
\end{definition}

Given a residuated distributive modal algebra $\mathcal{M} = \langle \mathcal{R}, \Diamond, \Box \rangle$ on a bounded distributive lattice $\mathcal{R}$, we define the set of all prime filters $\operatorname{PF}(\mathcal{L})$ and a map $\phi$ similarly to the bounded distributive lattice case described in the paper appendix. Let us define a ternary relation $R \subseteq \operatorname{PF}(\mathcal{L}) \times \operatorname{PF}(\mathcal{L}) \times \operatorname{PF}(\mathcal{L})$ as
$R(A, B, C) \Leftrightarrow A \bullet B \subseteq C$ and $E = \phi(\varepsilon)$. We also define binary relations $R_{\Box}$ and $R_{\Diamond}$ on $ \operatorname{PF}(\mathcal{L})$ as $A R_{\Diamond} B \Leftrightarrow b \in B \Rightarrow \Diamond b \in A$ and
$A R_{\Box} B \Leftrightarrow \Box a \in A \Rightarrow a \in B$.

Standardly, the subbasis of the topology $\tau$ is defined with by the sets $\phi(a)$ and $- \phi(a)$, $a \in \mathcal{M}$. Then the structure $\langle \operatorname{PF}(\mathcal{M}), \tau, E, R, R_{\Box}, R_{\Diamond} \rangle$ is the \emph{dual space} of a residuated distributive modal algebra $\mathcal{M}$.

Let $\mathcal{X} = \langle X, \tau, \leq, R, R_{\Box}, R_{\Diamond}, E \rangle$ be a modal bDRL-space and
$\operatorname{ClUp}(\mathcal{X})$ the set of all upwardly closed clopens of $\mathcal{X}$. We define product
as a binary operation on $\operatorname{ClUp}(\mathcal{X})$ as $A \circ B = R [A, B, \rule{0.25cm}{0.15mm}]$. Residuals are $A \setminus B = \{ c \in X \: | \: A \circ c \subseteq B \}$ and $B / A = \{ c \in X \: | \: c \circ A \subseteq B \}$. Modal operators $[R_{\Box}]$, $R_{\Diamond}$ are defined as $[R_{\Box}] A = \{ a \in X \: | \: \forall b \in X (a R_{\Box} b \Rightarrow b \in A)\}$
and $\langle R_{\Diamond} \rangle A = \{ a \in X \: | \: \exists b \in X (a R_{\Diamond} b \: \& \: b \in A)\}$.
The structure $\langle \operatorname{ClUp}(\mathcal{X}), \cup, \cap, \emptyset, X, \circ, \setminus, /, E, [R_{\Box}], \langle R_{\Diamond} \rangle \rangle$ is the \emph{dual algebra} of a modal bDRL-space $\mathcal{X}$. One may show that $\phi$ commutes with products and residuals \cite{galatos2003varieties}, that is:

$\phi$ commutes with modal operators as $\phi(\Diamond a) = \langle R_{\Diamond} \rangle [\phi(a)]$ and $\phi(\Box a) = [R_{\Box}] [\phi(a)]$ that also may be shown similarly to Proposition 5.2.1 here \cite{palmigiano2004dualities}.

The key theorem is the following one:

\begin{theorem}
$ $

\begin{enumerate}
\item The dual algebra of a modal bDRL-space $\mathcal{X}$ is an RDMA.
\item The dual space of a residuated distributive modal algebra $\mathcal{M}$ is a modal bDRL-space.
\end{enumerate}
\end{theorem}

\begin{proof}

\begin{enumerate}
  \item Let $A, B$ be upwardly closed clopens, let us show that $[R_{\Box}] A \bullet [R_{\Box}] B \subseteq [R_{\Box}] (A \bullet B)$.
  Let $c \in [R_{\Box}] A \bullet [R_{\Box}] B$, then $c \in R[R_{\Box}A, [R_{\Box}] B, \rule{0.25cm}{0.15mm}]$.
  This denotes that $R (a,b,c)$ for some $a \in R_{\Box} A$ and $b \in R_{\Box} B$. Let $c' \in X$ such that $c R_{\Box} c'$, let us show that $c' \in A \bullet B$.
  $R(a, b, c)$ and $c R_{\Box} c'$ implies that $R(a', b', c')$ for some $a' \in R_{\Box} (a)$ and $b' \in R_{\Box} (b)$ by the definition of a modal bDRL-space. Thus, $a' \in A$ and $b' \in B$, so $c' \in A \bullet B$.
  \item We check the condition that connect the ternary relation with the $\Box$-relation. Let $A, B$ be prime filters. Let us suppose that $R A B C$ and $C R_{\Box} C'$, where $A, B, C, C'$ are prime filters. Let us show that $R A' B' C'$ for some prime filters $A', B'$. Let us put $A_1 := \{ a \in X \: | \: \Box a \in A \}$ and $B_1 := \{ b \in X \: | \: \Box b \in B \}$. $A_1$ and $B_1$ are clearly filters. Now let us show that $A_1 \cdot B_1 \subseteq C'$. Let $a \cdot b \in A_1 \cdot B_1$, then $\Box a \in A$ and $\Box b \in B$. So $\Box a \cdot \Box b \in A \bullet B$, so $\Box (a \cdot b) \in A \bullet B$. Hence $\Box (a \cdot b) \in C$ and $a \cdot b \in C'$. By Lemma~\ref{filter}, there exist prime filters $A' \supseteq A_1$ and $B' \supseteq B_1$ such that $A'' B'' \subseteq C'$.

Let $A$ be a prime filter, then $R_{\Box} (A)$ and $R_{\Diamond}(A)$ are closed that might be shown as the well-known Esakia's lemma \cite{esakia1974topological}. Let $A$ be an upwardly closed clopen, $R_{\Box}(A)$ and $R_{\Diamond} (A)$ are upwardly closed clopens. The last two facts might be proved similarly to the intutionistic modal logic case \cite{palmigiano2004dualities}.
\end{enumerate}
\end{proof}

\begin{definition}
Given modal bDRL spaces $\mathcal{X}$, $\mathcal{Y}$, a contnuous bounded morphism is a map
$f : \mathcal{X} \to \mathcal{Y}$ such that $f$ is a Priestley map that preserves ternary and binary relations as a bounded morphism.
\end{definition}

As usual, the previous lemma allows one to claim that $\operatorname{PF}$ and $\operatorname{ClUp}$ are contravariant functors. The following theorem establish a desired topological duality itself.

\begin{theorem}
$ $

\begin{enumerate}
\item Let $\mathcal{M}$ be a residuated distributive modal algebra, then
$\operatorname{ClUp}(\operatorname{PF}(\mathcal{M})) \cong \mathcal{M}$
\item Let $\mathcal{X}$ be a modal dBRL-space, then
$\operatorname{PF}(\operatorname{ClUp}(\mathcal{X})) \cong \mathcal{X}$
\item Contravariant functors $\operatorname{PF} : \operatorname{RDMA} \rightleftarrows \operatorname{TKF} : \operatorname{ClUp}$ constitute a dual equivalence between the category of all substructural distributive modal algebras and the category of all modal dBRL-spaces.
\end{enumerate}
\end{theorem}

\begin{proof}
$ $

(i) The isomorphism is map $a \mapsto \phi(a)$ that commutes with products, residuals and modal operators as discussed above.

(ii) A homeomorphism is a map $g : x \mapsto \{ A \in \operatorname{ClUp} \: | \: x \in A \}$. As it is shown by Galatos,
$R_{\mathcal{X}}(x,y,z) \Leftrightarrow R_{\operatorname{PF}(\operatorname{ClUp}(\mathcal{X}))}(g(x), g(y), g(x))$.
One may immediately extend this homeomorphism and show that this map commutes with binary modal relations.

(iii) Follows from the previous two items and the previous theorem. Let us ensure briefly that these functors behave as expected with morphisms. Let $h : \mathcal{M}_1 \to \mathcal{M}_2$ be an RDMA homomorphism, then a map
$\operatorname{PF}(h) : \operatorname{PF}(\mathcal{M}_2) \to \operatorname{PF}(\mathcal{M}_1)$ such that
$\operatorname{PF}(h) : F \mapsto h^{-1} [F]$ is a Priestley map, where $F$ is a prime filter. This map also satisfies the monotonicity and
lifting properties for a ternary relation, see \cite{galatos2003varieties}. One may show that $\operatorname{PF}(h)$ has the monotonicity and lifting properties for $[R_{\Box}]$ and $\langle R_{\Diamond} \rangle$ similarly to the intutionistic modal logic case. On the other hand, $\operatorname{ClUp}(f) : \operatorname{ClUp}(Y) \to \operatorname{ClUp}(X)$ such that
$\operatorname{ClUp}(f) : A \mapsto f^{-1}[A], A \in \operatorname{ClUp}(Y)$ is a lattice homomorphism. $\operatorname{ClUp}(f)$
also preserves product and residuals. $\operatorname{ClUp}(f)$ preserves $[R_{\Box}]$ and $\langle R_{\Diamond} \rangle$
similarly to the intuitionistic modal logic case.
\end{proof}

Hence, we obtained the Priestley-style duality for the category of residuated distributive modal algebras and the category of all topological Kripke frames introduced by us.

\section{Further work}

In this paper, we examined canonicity for the distributive full Lambek calculus and its modal extensions within a ``usual'' Kripkean semantics. The further questions that should be solved are Sahlqvist and Goldblatt-Thomason theorems for such semantics and its non-distributive generalisation to study canonicity and modal definability for noncommutative modal logic with residuals in depth. One may consider for these purposes the frameworks described in \cite{conradie2018goldblatt} and \cite{DBLP:journals/apal/ConradieP19}.

One may also consider a Kleene star as a modal operator \cite{kuznetsov2018continuity}, but such a modality is neither $\Box$ nor $\Diamond$. The (non)canonicity of the variety of residuated Kleene lattices also should be studied and explored considering a Kleene residuated lattice as a sort of bounded lattice with operators.

We took the Lambek calculus with additive connections as the underlying logic requiring the lacking distributivity principle which is unprovable in the full Lambek calculus. A Priestley-style topological duality provided in this paper might be extended considering dual spaces for non-distributive residuated modal algebras using the canonical extensions technique studied by Gehrke and van Gool, e.g., here \cite{gehrke2014distributive}.

The distributive Lambek calculus also has a sort of Gentzen cut-free calculus and FMP \cite{kozak2009distributive}, but the same issues for the modal extensions are not investigated yet.

\Appendix

In this section, we recall the required background related to bounded distributive lattice canonical extensions and topological duality. Such notions as lattice, distributive lattice, filter, and prime filters are supposed to be known. We refer the reader to these textbooks
\cite{davey2002introduction} \cite{Sank}.

Let us recall Priestley duality, the dual equivalence between the caterory of bounded distributive lattices and the category of Priestley spaces \cite{priestley1972ordered}. First of all, we define a Priestley space.

\begin{definition}
  $ $

  \begin{enumerate}
  \item A Priestley space is a triple $\mathcal{X} = \langle X, \tau, \leq \rangle$, where $\langle X, \tau \rangle$ is a compact topological space, $\langle X, \leq \rangle$ is a bounded partial order with the additional Priestley separation axiom:
  \begin{center}
    $x \not\leq y \Rightarrow$ there exists a clopen $U \subseteq X$ such that $x \in U$ and $y \not\in U$
  \end{center}
  \item Let $\mathcal{X}_1, \mathcal{X}_2$ be Priestley spaces and $f : \mathcal{X}_1 \to \mathcal{X}_2$, then $f$ is a Priestley map, if $f$ is continuous, order-preserving and preserves bounds.
 \end{enumerate}
\end{definition}

Let $\mathcal{L}$ be a bounded distributive lattice and $\operatorname{PF}$ the set of prime filters in $\mathcal{L}$. Let us define a map $\phi : \mathcal{L} \to \mathcal{P}(\operatorname{PF}(\mathcal{L}))$ such that $\phi : a \mapsto \{ F \in \operatorname{PF}(\mathcal{L}) \: | \: a \in F \}$. The sets $\phi(a)$ and $- \phi(a)$ form a subbasis of topology on $\operatorname{PF}(\mathcal{L})$, where $a \in \mathcal{L}$. The structure $\langle \operatorname{PF}(\mathcal{L}), \tau, \subseteq \rangle$ is a Priestley space, where $\tau$ is generated by the subbasis above.

Let $\mathcal{X} = \langle X, \tau, \leq \rangle$ be a Priestley space and $\operatorname{ClUp}(\mathcal{X})$ the set of all clopen upwardly closed subsets of $\mathcal{X}$. The dual algebra of a Priestley space is an algebra $\langle \operatorname{ClUp}(\mathcal{X}), \cap, \cup, \emptyset, X \rangle$, which is a bounded distribute lattice.

Let $\mathcal{L}$ be a bounded distributive lattice, then
$\eta_{\mathcal{L}} : \mathcal{L} \to \operatorname{ClUp}(\operatorname{PF}(\mathcal{L}))$ is a lattice isomorphism. Given a Priestley space $\mathcal{X}$, then $\varepsilon_{\mathcal{X}} : \mathcal{X} \to \operatorname{PF}(\operatorname{ClUp}(\mathcal{X}))$ is a Priestley homeomorphism. Moreover, if $h : \mathcal{L}_1 \to \mathcal{L}_2$ is a bounded lattice homomorphism,
then $\operatorname{PF}(h) : \operatorname{PF}(\mathcal{L}_2) \to \operatorname{PF}(\mathcal{L}_1)$ is a Priestley map, where $\operatorname{PF}(h) : A \mapsto h^{-1}(A)$. If $f : \mathcal{X}_1 \to \mathcal{X}_2$ is a Priestley map, then $\operatorname{ClUp}(f) : \operatorname{ClUp}(\mathcal{X}_2) \to \operatorname{CU}(\mathcal{X}_1)$ is a bounded lattice homomorphism, where $\operatorname{ClUp}(f) : A \mapsto f^{-1}(A)$. The facts mentioned above establish Priestley duality:

\begin{theorem}
  The functors $\operatorname{PF} : {\bf BDistr} \rightleftarrows {\bf Priest} : \operatorname{ClUp}$ constitute a dual equivalence between the category of all bounded distributive lattices and the category of all Priestley spaces.
\end{theorem}

  \vspace{\baselineskip}

Canonical extensions were introduced by Jonsson and Tarski to extend a Stone representation to Boolean algebras with operators \cite{BAOs}.
Let us overview canonical extensions of distributive lattice expansions. We refer the reader to these paper \cite{gehrke1994bounded} \cite{Gehrke_Jonsson_2004} to have a more detailed picture of bounded distributive lattice expansions and canonical extensions for them.

Given a complete lattice $\mathcal{L}$, $a \in \mathcal{L}$ is called \emph{completely join irreducible},
if $a = \bigvee \limits_{i \in I} a_i$ implies that $a = a_i$ for some $i \in I$. Completely meet irreducible elements are defined dually. By $J^{\infty}(\mathcal{L})$ ($M^{\infty}(\mathcal{L})$ we denote the set of all completely join (meet) irreducible elements. There is an order isomorphism $\kappa : J^{\infty}(\mathcal{L}) \to M^{\infty}(\mathcal{L})$ such that $\kappa : u \mapsto \bigvee (- \uparrow u)$ \cite{gehrke1994bounded}.

A complete lattice $\mathcal{L}$ is \emph{completely distributive} \cite{davey2002introduction}, if for each doubly indexed subset $\{ x_{ij} \}_{i \in I, j \in J}$ of $\mathcal{L}$ one has
\begin{center}
  $\bigwedge_{i \in I} (\bigvee_{j \in J} x_{ij}) = \bigvee_{\alpha : I \to J} (\bigwedge_{i \in I} x_{i\alpha(i)})$
\end{center}
A perfect distributive lattice is a doubly algebraic completely distributive lattice, that is:
\begin{definition} \label{cdlattice}
  Let $\mathcal{L}$ be a bounded distributive lattice, then $\mathcal{L}$ is called perfect distributive lattice, if it is completely distributive and every $x \in \mathcal{L}$ has the form $x = \bigvee \{ j \in J^{\infty}(\mathcal{L}) \: | \: j \leq x\}$ and
  $x = \bigwedge \{ m \in M^{\infty}(\mathcal{L}) \: | \: x \leq m \}$. That is,
  $J^{\infty}(\mathcal{L})$ and $M^{\infty}(\mathcal{L})$ are join-dense
  and meet-dense in $\mathcal{L}$ correspondingly.
\end{definition}

Given a lattice $\mathcal{L}$, by the completion of $\mathcal{L}$ we mean an embedding
$\mathcal{L} \hookrightarrow \mathcal{L}'$, where $\mathcal{L}'$ is a complete lattice. For simplicity, we assume that $\mathcal{L}'$ contains $\mathcal{L}$ as a sublattice.
The definition of a canonical extension is the standard one:
\begin{definition}
  Let $\mathcal{L}$ be a bounded distributive lattice, a canonical extension of $\mathcal{L}$ is a completion
  $\mathcal{L} \hookrightarrow \mathcal{L}^{\sigma}$, where $\mathcal{L}^{\sigma}$ is
  a complete lattice and the following conditions hold:
  \begin{enumerate}
    \item (Density) Every element of $\mathcal{L}^{\sigma}$ is both a join of meets and meets of joins of elements from $\mathcal{L}$
    \item (Compactness) Let $S, T \subseteq \mathcal{L}$ such that $\bigwedge S \leq \bigvee T$ in $\mathcal{L}^{\sigma}$, then there exist finite subsets $S' \subseteq S$ and $T' \subseteq T$ such that $\bigwedge S' \leq \bigvee T'$.
  \end{enumerate}
\end{definition}
Now we define filter and ideal elements, or, closed and open elements, according to the alternative terminology.
\begin{definition}
    Let $\mathcal{L}$ be a bounded distributive lattice and $\mathcal{L}^{\sigma}$ a canonical extension of $\mathcal{L}$. Let us define the following sets:
    \begin{enumerate}
      \item $\mathcal{F}(\mathcal{L}^{\sigma}) = \{ x \in \mathcal{L}^{\sigma} \: | \: \text{$x$ is a meet of elements from $\mathcal{L}$}\}$, the set of filter elements
      \item $\mathcal{I}(\mathcal{L}^{\sigma}) = \{ x \in \mathcal{L}^{\sigma} \: | \: \text{$x$ is a join of elements from $\mathcal{L}$}\}$, the set of ideal elements
    \end{enumerate}
\end{definition}
It is known that the poset $F(\mathcal{L}^{\delta})$ is isomorphic to the poset $\operatorname{Filt}(\mathcal{L})$, the set of all filters of $\mathcal{L}$ and the similar statement holds for $\mathcal{I}(\mathcal{L}^{\sigma})$ and the of all ideals of $\mathcal{L}$.
We recall that a canonical extension of a bounded lattice $\mathcal{L}$ is unique up to isomoprphism that fixes $\mathcal{L}$, see, e. g., \cite{Expan}. For each canonical extension $\mathcal{L} \hookrightarrow \mathcal{L}^{\sigma}$, the poset $\mathcal{F}(\mathcal{L}^{\sigma}) \cup \mathcal{I}(\mathcal{L}^{\sigma})$ is uniquely defined by $\mathcal{L}$. The uniqueness of a canonical extension of $\mathcal{L}$ up to an isomorphism fixing $\mathcal{L}$ follows from this observation.

In the case of bounded distributive lattices, one has the following fact \cite{gehrke1994bounded}:
\begin{prop}
  $ $

  Let $\mathcal{L}$ be a bounded distributive lattice, then $\mathcal{L}^{\sigma}$ is a perfect distributive lattice.
\end{prop}

We also note that canonical extensions commute with dual order and Cartesian product. That is, $(\mathcal{L}^{op})^{\sigma} \cong (\mathcal{L}^{\sigma})^{op}$ and $(\mathcal{L}^n)^{\sigma} \cong (\mathcal{L}^{\sigma})^n$. Now we overview bounded distributive lattices expansions, that is,
bounded distributive lattices enriched with the additional family of operators, and their canonical extensions \cite{Expan}.

  Let $\mathcal{K}$, $\mathcal{L}$ be bounded distributive lattices and
$f : \mathcal{K} \to \mathcal{L}$.
Let us define maps $f^{\sigma}$, $f^{\pi} : \mathcal{K}^{\sigma} \to \mathcal{L}^{\sigma}$
as follows:
\begin{enumerate}
    \item $f^{\pi}(u) = \bigwedge \{ \bigvee \{ f(a) \: | \: a \in \mathcal{K}, x \leq a \leq y \} \: | \: \mathcal{F}(\mathcal{K}^{\sigma}) \ni x \leq u \leq y \in \mathcal{I}(\mathcal{K}^{\sigma}) \}$
    \item $f^{\sigma}(u) = \bigvee \{ \bigwedge \{ f(a) \: | \: a \in \mathcal{K}, x \leq a \leq y \} \: | \: \mathcal{F}(\mathcal{K}^{\sigma}) \ni x \leq u \leq y \in \mathcal{I}(\mathcal{K}^{\sigma}) \}$
\end{enumerate}
Every element of $\mathcal{K}$ is a filter element and an ideal element of $\mathcal{L}^{\sigma}$,
then $f^{\sigma}$ and $f^{\pi}$ both extend $f$. It is clear that $f^{\sigma} \leq f^{\pi}$ in a pointwise order. We formulate the following fact about extensions of maps on bounded distributive lattices that we are going to use further. If the original map $f$ is order preserving, then one may simplify the definitions of $f^{\sigma}$ and $f^{\pi}$ as follows,
e.g. \cite{Distr} \cite{Expan}:
\begin{prop}
  $ $

  Let $\mathcal{K}$, $\mathcal{L}$ be bounded distributive lattices and $f : \mathcal{K} \to \mathcal{L}$. Then:
  \begin{enumerate}
    \item $f^{\pi}(u) = \bigwedge \{ \bigvee \{ f(a) \: | \: a \in \mathcal{K}, a \leq y \} \: | \: u \leq y \in \mathcal{I}(\mathcal{K}^{\sigma}) \}$, where $u \in \mathcal{K}^{\sigma}$
    \item $f^{\sigma}(u) = \bigvee \{ \bigwedge \{ f(a) \: | \: a \in \mathcal{K}, x \leq a \} \: | \: u \geq x \in \mathcal{F}(\mathcal{K}^{\sigma}) \}$, where $u \in \mathcal{K}^{\sigma}$
    \item $f^{\sigma}(x) = f^{\pi}(x) = \bigwedge \{ f(a) \: | \: x \leq a \in \mathcal{K} \}$, where $x \in \mathcal{F}(\mathcal{K}^{\sigma})$
    \item $f^{\sigma}(y) = f^{\pi}(y) = \bigvee \{ f(a) \: | \: a \in \mathcal{K}, a \leq y \}$, where $y \in \mathcal{I}(\mathcal{K}^{\sigma})$
  \end{enumerate}
\end{prop}
As a consequence, for each $u \in \mathcal{K}^{\sigma}$ one has $f^{\sigma}(u) = \bigvee \{ f^{\sigma}(x) \: | \: u \geq x \in \mathcal{F}(\mathcal{K}^{\sigma}) \}$.
The third and fourth items denote that $f^{\sigma}$ and $f^{\pi}$ send filter (ideal) elements to filter (ideal) ones. A map $f : \mathcal{K} \to \mathcal{L}$ is called \emph{smooth}, if $f^{\pi}(u) = f^{\sigma}(u)$ for each $u \in \mathcal{K}^{\sigma}$. In particular, a map is smooth if it preserves or reverses finite joins or meets.
For instance, modal operators $\Box$ and $\Diamond$ are smooth and their smoothness since $\Box$ and $\Diamond$ are meet and join hemimorphisms correspondingly. Thus, $\Box^{\sigma} = \Box^{\pi}$ and $\Diamond^{\sigma} = \Diamond^{\pi}$.



\bibliographystyle{aiml20}
\bibliography{Rogozin}

\end{document}